\documentclass[11pt,a4paper]{article}
\usepackage{a4wide,amsfonts,amsmath,latexsym,amsthm,amssymb,euscript,eufrak,graphicx,units,mathrsfs,setspace,stmaryrd}
\usepackage[pdftex]{color}
\definecolor{darkblue}{rgb}{0.3,0.3,0.7}

\usepackage{float}
\newfloat{figure}{H}{lof}
\floatname{figure}{\figurename}

\usepackage[french,english]{babel}
\usepackage[T1]{fontenc}
\usepackage[utf8]{inputenc}
\usepackage{empheq}

\DeclareMathAlphabet{\eufrak}{U}{}{}{}  
\SetMathAlphabet\eufrak{normal}{U}{euf}{m}{n}
\SetMathAlphabet\eufrak{bold}{U}{euf}{b}{n}

\newtheorem{prop}{Proposition}[section]

\newtheorem{theorem}[prop]{Theorem}
\newtheorem{lemma}[prop]{Lemma}

\theoremstyle{definition}

\newtheorem{remark}[prop]{Remark}

\newtheorem{definition}[prop]{Definition}

\newtheorem{notation}[prop]{Notation}

\numberwithin{equation}{section}

\def\E{\mathbb{E}}
\def\P{\mathbb{P}}
\def\real{\mathbb{R}}

\def\1{\textbf{1}}
\def\ind#1{\textbf{1}_{\left\{#1\right\}}}

\def\X{\mathbb{X}}

\def\1{\textbf{1}}
\def\ind#1{\textbf{1}_{\{#1\}}}

\def\E{\mathbb{E}}
\def\P{\mathbb{P}}
\def\real{\mathbb{R}}

\def\T{{T \wedge \tau}}

\def\bN{\bold{N}}
\def\btN{\bold{\tilde{N}}}

\def\bO{\boldsymbol{\omega}}
\def\bX{\bold{X}}

\def\br{\boldsymbol{\rho}}
\def\bx{\bold{x}}
\def\by{\bold{y}}
\def\bD{\bold{D}}
\def\bpI{\boldsymbol{\mathcal I}}
\def\T{\mathcal T}

\newcommand{\be}{\begin{equation}}
\newcommand{\ee}{\end{equation}}
\newcommand{\bde}{\begin{displaymath}}
\newcommand{\ede}{\end{displaymath}}
\newcommand{\beq}{\begin{eqnarray*}}
\newcommand{\eeq}{\end{eqnarray*}}
\newcommand{\beqa}{\begin{eqnarray}}
\newcommand{\eeqa}{\end{eqnarray}}
\newcommand{\bel }{\left\{\begin{array}{ll}}
\newcommand{\eel}{\cr \end{array} \right.}

\newcommand{\bex}{\begin{ex} \rm }
\newcommand{\eex}{\end{ex}}
{

\def\E{\mathbb E}

\def\P{\mathbb P}

\definecolor{ying}{rgb}{0.8, 0.0, 0.04}

\DeclareSymbolFontAlphabet{\mathrsfs}{rsfs}

\author{Caroline Hillairet\footnote{ENSAE  Paris, CREST UMR 9194,
5  avenue Henry Le Chatelier
91120 Palaiseau, France.  Email: \texttt{caroline.hillairet@ensae.fr}} \and Thomas Peyrat\footnote{ENSAE  Paris, CREST UMR 9194, and Exiom Parners,  26 rue Notre Dame des Victoires, 75002 Paris, France.  \; Email: \texttt{thomas.peyrat@ensae.fr}} \and Anthony R\'eveillac\footnote{INSA de Toulouse, IMT UMR CNRS 5219, Universit\'e de Toulouse, 135 avenue de Rangueil 31077 Toulouse Cedex 4 France. \; Email: \texttt{anthony.reveillac@insa-toulouse.fr}} }

\title{ Poisson imbedding meets the Clark-Ocone formula  \footnote{{Caroline Hillairet's research benefited from the support of  the Chair Stress Test, RISK Management and Financial Steering , led by the French Ecole Polytechnique and its Foundation and sponsored by BNP Paribas.}}}

\begin{document}

\maketitle

\allowdisplaybreaks

\begin{abstract}
\noindent In this paper we develop a representation formula of Clark-Ocone type for any integrable Poisson functionals, which extends the Poisson imbedding for point processes. This representation formula differs from the classical Clark-Ocone formula on three accounts. First the representation holds with respect to the Poisson measure instead of the compensated one;  {second} the representation holds true in $L^1$ and not in $L^2$; and finally contrary to the classical Clark-Ocone formula the integrand is defined as a pathwise operator and not as a $L^2$-limiting object. We make use of  Malliavin's calculus and of the pseudo-chaotic decomposition with uncompensated iteraded integrals 
to establish  this Pseudo-Clark-Ocone representation formula and to characterize  the integrand,   which turns out to be  a predictable integrable process.
\end{abstract}


\section{Introduction}
\label{section:intro}
Martingale representation formulas for Poisson functionals {have} been widely investigated in the literature (such as L{\o}kka \cite{lokka2004martingale}, Nualart and Schoutens \cite{nualart2000chaotic}) {in relation with stochastic analysis tools such as the chaotic expansion or the Malliavin calculus.} At the crossroad of martingale representation and the Malliavin calculus lies the so-called Clark-Ocone\footnote{also referred as Clark-Ocone-Haussmann formula} formula which allows one to provide a description of the integrand in terms of the Malliavin derivative. More precisely, on the Poisson space, the Clark-Ocone formula gives the following representation of a square-integrable random variable $F$ :
$$ F = \E[F] + \int_{\real_+ }    ( D_{t}  F  )^p \; \tilde{N}(dt)$$
where $ \tilde{N}$  is the compensated Poisson process and  $ (D_{t} F)^p$ is the predictable projection of the Malliavin derivative of $F$. {It is important to note that in nature the Clark-Ocone formula and the martingale representation slightly differ as the latter may exist when the former fails to hold (like for instance for the fractional Brownian motion which is not a semimartingale with respect to its natural filtration). In addition the Clark-Ocone formula is restricted to representation of random variables whereas martingale representation provides a dynamic representation of the integrand for a given martingale. Yet in a Poisson framework both representations are intricated.}
Various kinds of generalizations have been obtained.
Using the Malliavin integration by parts formula given in Picard \cite{Picard_French_96}, Zhang  \cite{zhang2009clark} establishes the following Clark-Ocone formula: for any bounded Poisson functional $F$ 
$$ F = \E[F] + \int_{\real_+ \times \X}    \left({ \bD_{(t,x)} F} \right)^p\; \tilde{\bN}(dt,dx)$$
where $ \tilde{\bN}$  is the compensated Poisson random measure,  and $ \left({ \bD_{(t,x)} F} \right)^p$  is the predictable projection of the Malliavin derivative   of $F$.
Last and Penrose \cite{last2011martingale}  also give a Clark-Ocone type martingale representation formula when the underlying filtration is generated by a Poisson process on a measurable space.
Flint and Torrisi \cite{TorrisiCO} provide  a Clark-Ocone formula for   point processes on a finite interval possessing a conditional intensity.
Di Nunno and Vives \cite{nunno2017malliavin} develop a Malliavin-Skorohod type calculus for additive processes 
 and obtain a generalization of the Clark-Ocone formula for random variables in $L^1$ { whose integrand operator is supposed to be in $L^1$ as well.} 
At the same time, the {so-called} Poisson imbedding (see Jacod \cite[Chapter 4]{Jacod_1979} or Br\'emaud and Massouli\'e \cite{Bremaud_Massoulie}) provide{s} also a representation of any point  process with respect to an  uncompensated  Poisson measure. More precisely, 
if $H$ is a point process on $\real_+$ with {stochastic}  intensity $\lambda$, then   these pair of processes $(H,\lambda)$ can be represented on a probability space $(\Omega,\mathcal F,\P)$ supporting a random Poisson measure $\bN$ on $\real_+ \times \real_+$ and the following representation holds true : 
\begin{equation}\label{imbedding}
H_T = \int_{(0,T]\times \real_+} \ind{\theta\leq \lambda_t} \bN(dt,d\theta); \quad \forall T>0. 
\end{equation}
Yet  Hillairet and R\'eveillac \cite{Nouspseudochaotic}   investigate the so called pseudo-chaotic expansion which involves also non-compensated  iterated  integrals. This recent contribution leads the way to the study of a Clark-Ocone representation formula with respect to the non-compensated Poisson measure $ \bN$, which in a sense will generalize the Poisson imbedding relation \eqref{imbedding}.
The aim of this  paper is to  investigate which Poisson-functionals admit such  Pseudo-Clark-Ocone decomposition  and to  determine the integrand.\\\\
More precisely, we consider a   Poisson random measure  $\bN$  defined on  $(\bX,\mathfrak X):=([0,T] \times \mathbb X; \mathcal B(\real_+)\otimes \mathcal X)$  equipped with a non-atomic $\sigma$-finite measure $\br:=dt \otimes \pi$, and  where $(\mathbb X,\mathcal X)$  is a complete separable metric  space. 
The Pseudo-Clark-Ocone formula is  stated  in Theorem \ref{th:pseudoCOfinite} for any $F$ in $L^1(\Omega)$ 
(under the assumption that $\br(\bX)<+\infty$) :
\begin{equation}\label{PCO}
 F = F(\bO_\emptyset) + \int_{[0,T] \times \X} \mathcal{H}_{(t,x)} F \; \bN(dt,dx), \; \P-a.s..
 \end{equation}
The operator 
$\mathcal H F$ is defined as follows:
$$ \mathcal{H}: \begin{array}{ll} L^0(\Omega) &\to L^0(\Omega\times \bX) \\ F &\mapsto \mathcal{H}_{(t,x)} F  := \bD_{(t,x)}F \circ \tau_t  \end{array}$$
where $ \bD_{(t,x)}$ is the Malliavin derivative operator and 
$$
\tau_t: \begin{array}{ll}
\; \; \; \; \; \; \Omega &\to \; \; \; \; \; \;  \Omega\\
\bO =  \sum_{j=1}^{n} \delta_{{(t_j,x_j)}}  & \mapsto \bO_t:= \sum_{j=1}^{n} \delta_{{(t_j,x_j)}} \ind{t_j < t}.
\end{array} 
$$
Let us emphasize that this Pseudo-Clark-Ocone formula holds for any $F$ in $L^1(\Omega)$ and its integrand $\mathcal H F$ is well defined pathwise and is proved to belong to $ \mathbb L^1_{\mathcal P}$ (see Theorems \ref{th:pseudoCOfinite} and \ref{t{th:pseudoCOinfinite}}). This is a major difference with the standard Clark-Ocone formula that requires $F \in L^2(\Omega)$ and for which the integrand $\left({ \bD F} \right)^p$ is only defined as a limit in $L^2(\Omega\times \bX)$ (see Remark \ref{COlimitL2}).\\ 
 
The paper is organized as follows. Notations and the description of the Poisson space and elements of Malliavin's calculus are presented in Section \ref{sec:Malliavin},  as well as the operators and iterated integrals of the (pseudo)-chaotic expansions.  The main contribution is stated in Section \ref{sec:main}, which  recalls the standard and pseudo chaotic expansions and then derives  the Clark-Ocone formula with respect to the Poisson measure. Finally Section \ref{sseclemmata} gathers some technical lemmata. 

\section{Elements of Malliavin's calculus on the Poisson space}\label{sec:Malliavin}
We introduce in this section some elements and notions of stochastic analysis on a general Poisson space. All the elements presented in this section are taken from Last \cite{Last2016,Last:2011aa}.

\subsection{The Poisson space}
$\mathbb{N}^*:=\mathbb{N} \setminus \{0\}$ denotes the set of positive integers and for any finite set $S$,  $|S|$ denotes  its cardinal. 
We fix $(\mathbb X,\mathcal X)$ a complete separable metric space equipped with a non-atomic $\sigma$-finite measure $\pi$ and we set $(\bX,\mathfrak X):=([0,T] \times \mathbb X; \mathcal B([0,T])\otimes \mathcal X)$ equipped with the non-atomic $\sigma$-finite measure $\br:=dt \otimes \pi$ with $dt$ the Lebesgue measure on $\real_+$ and  $T$  is a fixed positive real number. Throughout this paper we will make use of the following notation : 
\begin{notation}
We denote with bold letters $\bx$ elements in $\bX$ and for $\bx \in \bX$ we set $\bx:=(t,x)$ with $t\in \real_+$ and $x\in \mathbb X$. 
\end{notation}
\noindent
We define $\Omega$ the space of configurations on $\bX$ as 
\begin{equation}
\label{eq:Omega}
\Omega:=\left\{\bO=\sum_{j=1}^{n} \delta_{\bx_j}, \; \bx_j \in \bX, \;j=1,\ldots,n, \; n\in \mathbb{N}\cup\{+\infty\} \right\}.
\end{equation}

\begin{notation}
\label{notation:emptyset}
We also denote by $\bO_\emptyset \in \Omega$  the unique element $\bO$ of $\Omega$ such that $\bO(A\times B)=0$ for any $A\times B \in \mathfrak X$.
\end{notation}

\noindent
Let $\mathcal F$ be the $\sigma$-field associated to the vague topology on $\Omega$. Let $\P$ the Poisson measure on $ \Omega$ under which 
the canonical evaluation $\bN$ defines a Poisson random measure with intensity measure $\br$. To be more precise given any element $A\times B$ in $ \mathfrak X$ with $\br(A\times B)>0$, the random variable
$$(\bN(\bO))(A\times B):= \bO(A\times B)$$
is a Poisson random variable with intensity $\br(A\times B)$  and $\btN$ defined as 
$$(\btN(\bO))(A\times B):= \bO(A\times B)-\br(A\times B)$$
is the compensated Poisson measure.
We define the natural history associated to $\bN$
$$ \mathcal F^{\bN}_t:=\sigma\{\bN(A\times B); (A,B) \in \mathfrak X; A\subset [0,t]\} $$
and $\mathcal{F}^{\bN}:=\sigma\{\bN(A\times B); (A,B) \in \mathfrak X\}$. 
For any $F \in L^1(\Omega,\mathcal F^{\bN})$ and any $t\geq 0$ we denote $ \E_t[F]:=\E[F\vert \mathcal F_t^{\bN}].$
We also set for any $t>0$,
$$ \mathcal F^{\bN}_{t-}:= \bigvee_{0<s<t} \mathcal F^{\bN}_{s}. $$
We set $\mathcal P^\bN$ the predictable $\sigma-$field associated to $\mathcal F^{\bN}$ that is the $\sigma$-field on $[0,T]\times \Omega$ generated by left-continuous $\mathcal F^{\bN}$-adapted stochastic processes and we set $\mathcal P:=\mathcal P^\bN \otimes \mathcal B(\X)$ the set of real-valued predictable processes $X:\Omega \times [0,T]\times \X$. Let 
$$ \mathbb L^0_{\mathcal P} :=\left\{X:\Omega \times [0,T]\times \X \textrm{ which are } \mathcal P\textrm{ -measurable} \right\},$$
and for $r > 0$
$$ \mathbb L^r_{\mathcal P} :=\left\{X \in \mathbb L^0_{\mathcal P}, \;\|X\|_{\mathbb L^r_{\mathcal P}}<+\infty\right\},
\mbox{  with  }  \|X\|_{\mathbb L^r_{\mathcal P}}:=\E\left[\int_{[0,T]\times \X} |X_{(t,x)}|^r dt\pi(dx)\right]^{1/r}.$$
According to Jacod \cite{Jacod_75} {(see also Protter \cite[Corollary 3]{Protter90})}, for any $X \in \mathbb L^2_{\mathcal P}$, the process $t\mapsto \int_{[0,t]\times \X} X_{(r,x)} \tilde{\bN}(dr,dx)$ is a $\mathcal F^\bN$-martingale which satisfies the following $L^2$-isometry
\begin{equation}
\label{eq:Isometry}
\E\left[\left|\int_{[0,t]\times \X} X_{(r,x)} \tilde{\bN}(dr,dx)\right|^2\right]=  \E\left[\int_{[0,t]\times \X} \left|X_{(r,x)}\right|^2 dr \pi(dx)\right], \; \forall t \in [0,T].
\end{equation}

\subsection{Add-points operators and the Malliavin derivative}

We introduce some elements of Malliavin's calculus on Poisson processes. We set 
$$ L^0(\Omega):=\left\{ F:\Omega \to \real, \; \mathcal{F}^{\bN}-\textrm{ measurable}\right\},$$
$$ L^2(\Omega):=\left\{ F \in L^0(\Omega), \; \E[|F|^2] <+\infty\right\}.$$
Similarly
$$L^0(\bX^j) := \left\{f:\bX^j \to \real, \; \mathfrak X^{\otimes j}-\textrm{measurable } \right\}$$
 and for $r\in\{1,2\}$, for $j\in \mathbb{N}^*$
\begin{equation}
\label{definition:L2j}
L^r(\bX^j) := \left\{f \in L^0(\bX^j), \; {\|f\|_{L^r(\bX^j)}^r:=} \int_{\bX^j} |f(\bx_1,\cdots,\bx_j)|^r \br^{\otimes j}(d\bx_1 \cdots d\bx_j) <+\infty\right\}.
\end{equation}
Besides,
\begin{equation}
\label{definition:L2js}
L^r_s(\bX^j) := \left\{f \in L^r(\bX^j) \textrm{ and } f \textrm{ is symmetric} \right\}
\end{equation}
is the space of square-integrable symmetric mappings where we recall that $f:\bX^j \to \real$ is said to be symmetric if for any element $\sigma$ in $\mathcal S_j$ (the set of all  permutations of $\{1,\cdots,j\}$), 
$$ f(\bx_1,\ldots, \bx_j) = f(\bx_{\sigma(1)},\ldots, \bx_{\sigma(j)}), \quad \forall (\bx_1,\ldots,\bx_j) \in \bX^j.  $$
Given $f:\bX^j \to \real$, we write $\tilde f$ its symmetrization defined as : 
$$ \tilde f(\bx_1,\ldots, \bx_j):= \frac{1}{n!} \sum_{\sigma \in \mathcal S_j} f(\bx_{\sigma(1)},\ldots, \bx_{\sigma(j)}).$$

\noindent
The main ingredient we will make use of are the add-points operators on the Poisson space $\Omega$. 

\begin{definition}$[$Add-points operators$]$\label{definitin:shifts}
Given $n\in \mathbb N^*$, and $J:=\{\bx_1,\ldots,\bx_n\} \subset \bX$ a subset of $\bX$ with $|J|=n$,
we set the measurable mapping :
\begin{eqnarray*}
\varepsilon_{J}^{+,n} : \Omega & \longrightarrow & \Omega  \\
     \bO & \longmapsto   & \bO + \sum_{\bx \in J} \delta_{\bx} \ind{\bO(\{\bx\})=0}.
\end{eqnarray*}
Note that by definition 
$$  \bO + \sum_{\bx \in J} \delta_{\bx} \ind{\bO(\{\bx\})=0} = \bO + \sum_{j=1}^{n} \delta_{\bx_j} \ind{\bO(\{\bx_j\})=0} $$
that is we add the atoms $\bx_j$ to the path $\bO$ unless they already were part of it (which is the meaning of the term $\ind{\bO(\{\bx_j\})=0}$). Note that since $\br$ is assumed to be atomless, given a set $J$ as above, $\P[\bN(J)=0]=1$ hence in what follows we will simply write $\bO + \sum_{j=1}^{n} \delta_{\bx_j}$ for $\varepsilon_{\bold{x}}^{+,n} (\bO)$.  
\end{definition}

\noindent We now recall the Malliavin derivative operator. 

\begin{definition}
\label{definition:Dn}
For $F$ in ${L^0(\Omega)}$, $n\in \mathbb N^*$, $(\bx_1,\ldots,\bx_n) \in \bX^n$, we set 
\begin{equation}
\label{eq:Dn}
\bold{D}_{(\bx_1,\ldots,\bx_n)}^{n} F:= \sum_{J \subset \{\bx_1,\ldots,\bx_n\}} (-1)^{n-|J|} F\circ \varepsilon_{J}^{+,|J|},
\end{equation}
where we recall that $\emptyset \subset \bX$.
For instance when $n=1$, we write $\bold{D}_{\bx} F := \bold{D}_{\bx}^1 F = F(\cdot + \delta_{\bx}) - F(\cdot)$ which is the difference operator (also called add-one cost operator\footnote{see \cite[p.~5]{Last2016}}). Relation (\ref{eq:Dn}) rewrites as 
$$ \bold{D}^n_{(\bx_1,\ldots,\bx_n)} F (\bO)= \sum_{J\subset \{1,\ldots,n\}} (-1)^{n-|J|} F\left(\omega + \sum_{j \in J} \delta_{\bx_j}\right), \quad \textrm{ for a.e. } \bO \in \Omega.$$
Note that with this definition, for any $\bO$ in $\Omega$, the mapping 
$$ (\bx_1,\ldots,\bx_n) \mapsto \bold{D}_{(\bx_1,\ldots,\bx_n)}^n F (\bO) $$
belongs to $L^0_s(\bX^j)$ defined as (\ref{definition:L2js}) and in addition the mapping 
\begin{equation}
\label{eq:Tn}
T_n F : (\bx_1,\ldots,\bx_n) \mapsto \E[\bold{D}_{(\bx_1,\ldots,\bx_n)}^n F] 
\end{equation}
is well-defined and belongs to $L^2_s(\bX^j)$ for any $F$ in $L^2(\Omega)$ (see \cite{Last2016,Last:2011aa}).
\end{definition}

\noindent
We recall the following property  (see \textit{e.g.} \cite[Relation (15)]{Last2016}) :
for  $F$ in $L^2(\Omega)$, $n\in \mathbb N^*$, and $(\bx_1,\ldots,\bx_n) \in \bX^n$, the $n$th iterated Malliavin's derivative operator $\bold{D}^n$ satisfies 
\begin{equation}\label{prop:Dnaltenative}
 \bold{D}^n F = \bold{D} (\bold{D}^{n-1} F), \quad n\geq 1; \quad \bold{D}^0 F:=F.
\end{equation}
\begin{remark}
\label{rk:Dderm}
If $F$ is deterministic, then by definition $\bD^n F=0$ for any $n\geq 1$.
\end{remark}

\subsection{Factorial measures and iterated integrals}

We first introduce a purely deterministic operator that  will be at the core of the pseudo-chaotic expansion. 

\begin{definition}
\label{definition:patwisederivative}
For $F\in L^0(\Omega)$, we define the deterministic operators: 
$$ \T_0 F:= F(\bO_\emptyset),$$ and for  $n\in \mathbb N^*$,  $(\bx_1,\cdots,\bx_n) \in \bX^n$,
$$ \T_n F (\bx_1,\cdots,\bx_n) := \sum_{J\subset \{\bx_1,\cdots,\bx_n\}} (-1)^{n-|J|} F\left(\sum_{\by \in J} \delta_{\by}\right).$$
\end{definition}

\noindent
In particular, even though $F$ is a random variable, $\T_n F{(\bx_1,\cdots,\bx_n)} $ is a real number as each term $F(\varpi_{J})$
 is the evaluation of $F$ at the outcome  $\varpi_{J}$
(where  $\varpi_{J}$ denotes $\sum_{\by \in J} \delta_{\by}$).
Besides,  given  the event $\{ \bN(\bX)=0 \}$, $\T_n F{(\bx_1,\cdots,\bx_n)} $ coincides with $  \bold{D}^n_{(\bx_1,\ldots,\bx_n)} F $ and  $ \T_0 F$ coincides with $F$.

\begin{prop}(Factorial measures; See \textit{e.g.} \cite[Proposition 1]{Last2016}).
\label{prop:factormeas}
There exists a unique sequence of counting random measures $(\bN^{(m)})_{m\in \mathbb N^*}$ where for any $m$, $\bN^{(m)}$ is a counting random measure on $(\bX^m, \mathfrak X^{\otimes m})$ with
\begin{align*}
&\bN^{(1)}:=\bN  \quad \mbox{ and \quad  for } \quad  A \in \mathfrak X^{m+1},\\
&\bN^{(m+1)}(A) \\ 
&:= \int_{\mathfrak X^m} \left[ \int_\mathfrak X \ind{(\bx_1,\ldots,\bx_{m+1})\in A} \bN(d\bx_{m+1}) - \sum_{j=1}^m \ind{(\bx_1,\ldots,\bx_{m},\bx_j)\in A} \right] \bN^{(m)}(d\bx_1,\ldots,d\bx_m); 
\end{align*}
\end{prop}
\noindent With this definition at hand we introduce the notion of iterated integrals. In particular for $A \in \mathfrak X$, 
$$ \bN^{(n)}(A^{\otimes n}) = \bN(A) (\bN(A)-1) \times \cdots \times (\bN(A)-n+1).$$
Note that by definition $\bN^{(n)}(A) \ind{\bN(A) < n}=0$. We now turn to the definition of two types of iterated integrals.

\begin{definition}
\label{definition:interated}
Let $n \in \mathbb N^*$ and $f_n \in L^1(\bX^n)$.
\begin{itemize}
\item 
$ \boldsymbol{\mathcal{I}}_n(f_n)$ the $n$th iterated integral of $f_n$ against the Poisson measure $\bN$ defined as 
$$ \boldsymbol{\mathcal{I}}_n(f_n):= \int_{\bX^{n}}  f_n(\bx_1,\ldots,\bx_n) \; \bN^{(n)}(d\bx_1,\ldots,d\bx_n),$$
where each of the integrals above is well-defined pathwise for $\P$-a.e. $\bO \in \Omega$.
\item $\bold{I}_n(f_n)$ the $n$th iterated integral of $f_n$ against the compensated Poisson measure $\tilde{\bN}$ defined as 
$$ \bold{I}_n(f_n) := \sum_{J\subset \{1,\ldots,n\}} (-1)^{n-|J|} \int_{\bX^{n-|J|}} \int_{\bX^{|J|}} f_n(\bx_1,\ldots,\bx_n) \; \bN^{(|J|)}(d\bx_J) \br^{\otimes (n-|J|)}(d\bx_{J^c}),$$
where $J^c:=\{1,\ldots,n\}\setminus J$ and $d\bx_J:=(d\bx_j)_{j\in J}$ and where each of the integrals above is well-defined pathwise for $\P$-a.e. $\bO \in \Omega$. We adopt the convention for $J=\emptyset$ that $\int ... \bN^{(0)}:=1$. It is worth noting that in the Poisson framework the iterated integrals $\bold{I}_n$ can be constructed in $L^1$ as integrals with respect to the factorial measures in a Lebesgue-Stieltjes fashion; or as the elements of the chaotic expansion which gives a construction of  $L^2(\Omega)$ as an orthonormal sum of subsets named chaos. We refer \textit{e.g.} to \cite{Privault_2009} for a description of this construction but we recall  in particular the isometry $\E[|\bold{I}_n(f_n)|^2] = \|f_n\|_{L^2(\bX^n)}^2$ whenever $f_n$ belongs to $L^2(\bX^n)$.\\\\
\noindent
Let us also remark that by definition $\bold{I}_n(f_n) = \bold{I}_n(\tilde f_n)$ and $\bpI_n(f_n) = \bpI_n(\tilde f_n)$ where $\tilde f_n$ denotes the symmetrization of $f_n$.

\end{itemize}
\end{definition}

\noindent To conclude this section, we recall a particular case of Mecke's formula (see \textit{e.g.} \cite[Relation (11)]{Last2016}).
\begin{lemma}[A particular case of Mecke's formula]
\label{lemma:Mecke}
Let $F \in L^0(\Omega)$, $k\in \mathbb N$ and $h\in L^0(\bX^k)$ then 
$$ \E\left[F \, \int_{\bX^k} h d\bN^{(k)}\right] = \int_{\bX^k} h(\bx_1,\ldots,\bx_k) \E\left[F\circ \varepsilon_{\bx_1,\ldots,\bx_k}^{+,k}\right] \br^{\otimes k}(d\bx_1,\ldots,d\bx_k),$$
providing the right-hand side is well defined as an integral in $L^1(\bX^k)$. The so-called integration by parts formula is then obtained with the same assumptions 
\begin{equation}
\label{eq:IPP}
\E\left[F \, \bold{I}_k(h) \right] = \int_{\bX^k} h(\bx_1,\ldots,\bx_k) \E\left[\bD_{(\bx_1,\ldots,\bx_k)}^k F\right] \br^{\otimes k}(d\bx_1,\ldots,d\bx_k).
\end{equation}
\end{lemma}

\section{Standard/Pseudo Chaotic expansion and Clark-Ocone formula}\label{sec:main}
 In the literature,   results on representations of Poisson functionals  are mainly concentrated on the chaos expansion involving iterated compensated integrals operators $\bold{I}_n$. Similarly  the classical Clark-Ocone formula is stated with respect to the compensated Poisson measure $\tilde \bN$.  Yet \cite{Nouspseudochaotic}   investigates the so called pseudo-chaotic expansion which involves iterated non-compensated integrals operators $\bpI_n$. This recent contribution leads the way to the study of a Clark-Ocone representation formula with respect to the non-compensated Poisson measure $ \bN$.

\subsection{Chaotic expansion and the Clark-Ocone formula}
Let us first  recall the classical chaotic expansion and  the Clark-Ocone formula with respect to  the compensated Poisson measure $\tilde \bN$.
\begin{theorem}[See \textit{e.g.} Theorem 2 in \cite{Last2016}]
\label{th:chaoticgeneral}
Let $F$ in $L^2(\Omega)$. Then there exists a unique sequence $(f_n^F)_{n \geq 1}$ with $f_n^F\in L_s^2(\bX^n)$ such that
$$ F = \E[F] + \sum_{n=1}^{+\infty} \frac{1}{n!} \bold{I}_n(f_n^F),$$
where the convergence of the series holds in $L^2(\Omega)$. In addition coefficients $(f_n^F)_n$ are given as 
$$ f_n^F = T_n F, \quad n\geq 1 $$
where $T_n F$ is defined by (\ref{eq:Tn}) and where the equality is understood in $L^2(\bX^n)$. In addition  Theorem 1 in \cite{Last2016} provides the convergence of  the series 
$$ \sum_{n=1}^{+\infty} \int_{\bX^n} \left|T_n F(\bx_1,\ldots,\bx_n)\right|^2 \br^{\otimes n}(d\bx_1,\ldots,d\bx_n) .$$
\end{theorem}

\noindent The other classical relation on the Poisson space is the Clark-Ocone formula, for which one can find in the literature different variants with different conditions, such as e.g.  Privault \cite{Privault_2009}, Zhang \cite{zhang2009clark}, Di Nunno and Vives \cite{nunno2017malliavin}. In our setting  it takes the following form.
\begin{theorem}[Clark-Ocone formula]
\label{th:CO}
Let $F$ in $L^2(\Omega)$. Then
$$ F = \E[F] + \int_{[0,T] \times \X} \left({ \bD_{(t,x)} F} \right)^p\tilde{\bN}(dt,dx)$$
where the process $\left({ \bD F} \right)^p$ belongs to $\mathbb L^2_{\mathcal P}$ and is in this context the predictable projection of  $\bD F$ which is properly defined in the Appendix.
\end{theorem}

\begin{remark}\label{COlimitL2}
Contrary to what it suggests, the meaning of the integrand in the stochastic integral has to be made precise. According to Definition \ref{definition:Dn}, $\bD F$ belongs to $L^0(\Omega\times \bX)$ whenever $F$ belongs to $L^0(\Omega)$;  however, the definition of the operator $\left({ \bD F} \right)^p$ is not guaranteed. For instance $F$ in $L^2(\Omega\times \bX)$ does not entail $\bD F \in L^2(\Omega\times \bX)$. Several sufficient conditions can be found in the literature (see Last \cite{Last2016}). In general as we will make it precise in the proof of Theorem \ref{th:CO} (see Section \ref{Appendix}), it is possible to define $\left({ \bD F} \right)^p$ in $L^2(\Omega\times \bX)$ in a limiting procedure and based on the continuous feature of the mapping $F \mapsto \left(\E_{t-}\left[\bD_{(t,x)} F\right]\right)_{(t,x)}$ together with the fact that for fixed $(t,x)$, $\E_{t-}\left[\bD_{(t,x)} F\right] = \left({ \bD_{(t,x)} F} \right)^p \; \mathbb P-a.s.$ (see \cite{zhang2009clark}). For sake of completeness, we reproduce the proof of this result in the Appendix (Section \ref{Appendix}) following \cite[Section 3.2]{Privault_2009}. {Note that  even in the case of the Clark-Ocone formula obtained in \cite{nunno2017malliavin} where the $L^2$ assumption on $F$ is relaxed to $L^1$, the integrand operator is assumed to belong to $L^1$ as well. In Theorem \ref{th:pseudoCOfinite}  below (Pseudo-Clark-Ocone formula) the integrability of the integrand is a consequence of $F$ in $L^1$.}
\end{remark}

\subsection{Pseudo-chaotic expansion and Pseudo-Clark-Ocone formula}

We recall  here the pseudo-chaotic expansion  obtained with respect to the non-compensated iterated integrals operators $\bpI_n$.
 \begin{theorem}[See Theorem 2 in \cite{Nouspseudochaotic}]\label{th:equipseudochao}
Assume  $\br(\bX)<+\infty$  and $F$ in $L^{2}(\Omega)$. Then 
\begin{equation}
\label{eq:pseudochaotic}
F=  F(\bO_\emptyset) + \sum_{k=1}^{+\infty} \frac{1}{k!} \bpI_k(\T_k F); \quad \P-a.s.
\end{equation}
\end{theorem}
\noindent Inspired by this pseudo-chaotic expansion, we introduce below the operator that will allow us to write a Clark-Ocone formula with respect to the non-compensated Poisson measure (that we called Pseudo-Clark-Ocone formula). 
\begin{definition}
\label{def:tau_t}
For fixed $t>0$, we consider the measurable transformation  $\tau_t$ 
$$
\tau_t: \begin{array}{ll}
\Omega &\to \Omega\\
\bO & \mapsto \bO_t 
\end{array} 
$$
where for $\bO$ of the form $\bO:=\sum_{j=1}^{n} \delta_{{(t_j,x_j)}}, \; (t_j,x_j) \in \bX, \;j=1,\ldots,n, \; n\in \mathbb{N}\cup\{+\infty\}$, we set 
$$ \bO_t:= \sum_{j=1}^{n} \delta_{{(t_j,x_j)}} \ind{t_j < t}.$$
We also set $\tau_0$ as 
$$ \tau_0(\bO):=\bO_\emptyset; \quad \bO\in \Omega$$
where we recall Notation \ref{notation:emptyset}.
\end{definition}
Roughly speaking, $\tau_t$ plays the role of a  $\mathcal F^{\bN}_{t^-}$-conditional expectation.

\begin{definition}
\label{definition:patwisederivativeconditionned}
Fix $(t,x)$ in $\bX$. For $F\in L^0(\Omega)$, we define the operator $\mathcal{H}$ as : 
$$ \mathcal{H}: \begin{array}{ll} L^0(\Omega) &\to L^0(\Omega\times \bX) \\ F &\mapsto \mathcal{H}_{(t,x)} F  := \bD_{(t,x)}F \circ \tau_t  \end{array}$$
that is  for $\bO \in \Omega$
$$ \mathcal{H}_{(t,x)} F (\bO) := \bD_{(t,x)}F(\bO_t) = F(\bO_t \circ \varepsilon^+_{(t,x)}) - F(\bO_t).$$
\end{definition}
{\noindent As highlighted in the following lemma,  it turns out that  the operator  $\mathcal{H}$ can be written as the combination of the Malliavin derivative with a very specific Girsanov transformation $L^{T,t}$. This is exactly the (non-equivalent) Girsanov transformation which is at the core of the pseudo-chaotic expansion (see \cite{Nouspseudochaotic}) and  under which the Poisson measure $\bN$  (with intensity $\br$ on $\bX$) becomes a Poisson measure with intensity 0 on $[t,T]\times \X$.
\begin{lemma}
\label{lemma:conditionning}
Assume $\br(\bX)<+\infty$ and $F \in L^0(\Omega)$. Fix $(t,x)\in (0,T)\times \X$. Then the operator $\mathcal{H}_{(t,x)}$ re-writes as 
$$ \mathcal{H}_{(t,x)} F = \E_{t-}[L^{T,t} \bD_{(t,x)} F], \; \P-a.s.$$
 $$ \mbox{ with } \quad L^{T,t}:= \exp((T-t)\pi(\X)) \ind{\bN([t,T]\times \X)=0}.$$
\end{lemma}
}
\medskip
\begin{proof}
Let $\bO \in \Omega$. 
\begin{align*}
(L^{T,t} \bD_{(t,x)} F)(\bO)
&=\exp((T-t)\pi(\X)) \ind{\bO([t,T]\times \X)=0} \bD_{(t,x)} F(\bO) \\
&=\left(\exp((T-t)\pi(\X)) \ind{\bN([t,T]\times \X)=0} \bD_{(t,x)} F\circ \tau_t \right)(\bO).
\end{align*}
Hence since $\bD_{(t,x)} F\circ \tau_t$ is $\mathcal F^{\bN}_{t-}$-mesurable 
\begin{align*}
\E_{t-}[L^{T,t} \bD_{(t,x)} F] &=\bD_{(t,x)} F\circ \tau_t \; \exp((T-t)\pi(\X)) \E_{t-}[\ind{\bN([t,T]\times \X)=0}]\\
& =(F(\bO_t \circ \varepsilon^+_{(t,x)}) - F(\bO_t)) \; \exp((T-t)\pi(\X)) \E_{t-}[\ind{\bN([t,T]\times \X)=0}] \\
&= \mathcal H_{(t,x)} F.
\end{align*}
\end{proof}

\noindent We have now all the elements to  state the Pseudo-Clark-Ocone formula  for any $F$ in $L^1(\Omega)$, first under the condition that $\br(\bX)<+\infty$.
\begin{theorem}
\label{th:pseudoCOfinite}
Assume $\br(\bX)<+\infty$. Let $F$ in $L^1(\Omega)$. Then $\mathcal H F\in \mathbb L^1_{\mathcal P}$ and
$$ F = F(\bO_\emptyset) + \int_{[0,T] \times \X} \mathcal{H}_{(t,x)} F \; \bN(dt,dx), \; \P-a.s..$$
In addition this decomposition is unique in the sense that if there exists $(c,Z) \in \mathbb R \times \mathbb L^1_{\mathcal P}$ such that 
\begin{equation}\label{proofUnicite}
F = c + \int_{[0,T] \times \X} Z_{(t,x)} \; \bN(dt,dx), \; \P-a.s..\end{equation}
then $c=F(\bO_\emptyset)$ and for a.-e. $(t,x) \in [0,T] \times \X$, $Z_{(t,x)} = \mathcal{H}_{(t,x)}F$, $\P$-a.s..
\end{theorem}
\begin{proof} To highlight the main keys  of the proof, some  technical lemmata (namely  Lemma \ref{lemma:existenceTtilde} and Lemma  \ref{lemma:span})  are postponed  in Section \ref{sseclemmata}.\\
\textbf{Uniqueness:} Assume Relation \eqref{proofUnicite} holds.  Then 
note that  
$ F(\bO_\emptyset) = F \ind{\bN(\bX)=0} = c.$\\
Fix $(t,x) \in \bX$, $t>0$. We have for $\P$-a.e. $\bO$ in $\Omega$, since $Z \in \mathbb L^1_{\mathcal P}$
\begin{align*}
\mathcal{H}_{(t,x)} F (\bO) 
&= \bD_{(t,x)}F(\bO_t) \\
&= \int_{[0,T]  \times \X} (Z_{(s,y)}(\bO_t+\delta_{(t,x)})-Z_{(s,y)}(\bO_t)) \; \bO_t(ds,dy) + Z_{(t,x)}(\bO_t+\delta_{(t,x)}) \\
&= Z_{(t,x)}(\bO_t+\delta_{(t,x)})
=Z_{(t,x)}(\bO_t)
=Z_{(t,x)}(\bO).
\end{align*}
\textbf{Existence:} We proceed in several parts. \\
 \underline{Step 1.} Since $\br(\bX)<+\infty$, $\P[\bN(\bX)=0]>0$. Let $F \in L^2(\Omega)$. Thus $ \lim_{p\to +\infty} \E[|F-F^k|^2] = 0$ with $F^k := \E[F] + \sum_{n=1}^k \bold{I}_n(T_n F)$. By Lemma \ref{lemma:span}  and the uniqueness of the pseudo-chaotic expansion, for any $k$,
$$ F^k = F^k(\bO_\emptyset) + \sum_{n=1}^k \bpI_n(\T_n F^k), 
\mbox{ and  }
 F^k = F^k(\bO_\emptyset)  + \int_{[0,T]\times \X} \mathcal H_{(t,x)} F^k \; \bN(dt,dx), $$
and $\mathcal H_{(t,x)} F^k$ belongs to $\mathbb L^1_{\mathcal P}$.
In addition by Lemma \ref{lemma:existenceTtilde}, $\mathcal H F$ is well-defined in $\mathbb L^1_{\mathcal P}$ and 
$$ \lim_{k\to+\infty} \E\left[\left|\int_{[0,T]\times \X}  \mathcal{H}_{(t,x)}F^k \bN(dt,dx) - \int_{[0,T]\times \X}  \mathcal{H}_{(t,x)} F \bN(dt,dx)\right|\right]=0.$$
Hence
\begin{align*}
&\E\left[\left|F - F(\bO_\emptyset) - \int_{[0,T]\times \X}  \mathcal{H}_{(t,x)} F \bN(dt,dx)\right|\right]\\
&\leq \E\left[\left|F^k - F^k(\bO_\emptyset) - \int_{[0,T]\times \X}  \mathcal{H}_{(t,x)} F^k \bN(dt,dx)\right|\right]\\ 
&+\E\left[\left|\int_{[0,T]\times \X}  (\mathcal{H}_{(t,x)} F^k-\mathcal{H}_{(t,x)} F) \bN(dt,dx)\right|\right]\\
&+\E\left[\left|F - F^k \right|\right]+\left|F(\bO_\emptyset) - F^k(\bO_\emptyset)\right| 
\quad \to_{k\to+\infty} 0
\end{align*}
 $$ \mbox{ as } \quad F(\bO_\emptyset) - F^p(\bO_\emptyset) = (\P[\bN(\bX)=0])^{-1} \E[|F - F^k| \ind{\bN(\bX)=0}] \to_{p\to+\infty} 0.$$
\underline{Step 2.} Assume $F \in L^1(\Omega)$. Let for $k\geq 1$, $F^k:=-k \vee F \wedge k $. Then by monotone convergence theorem $\lim_{k\to+\infty}\E[|F-F^k|] =0$. For any $k$, as $F^k \in L^2(\Omega)$, 
by Step 1, 
$$ F^k = F^k(\bO_\emptyset) + \int_{\bX} \mathcal H_{(t,x)} F^k \bN(dt,dx).$$
By definition, 
\begin{align*}
|F^k(\bO_\emptyset) - F(\bO_\emptyset)| &= ((\P[\bN(\bX)=0])^{-1}) \E[|F^k(\bO_\emptyset) -F(\bO_\emptyset)|] \\
&= ((\P[\bN(\bX)=0])^{-1}) \E[|F(\bO_\emptyset)| \ind{|F(\bO_\emptyset)| > k}] \\
&\to_{k\to+\infty} 0.
\end{align*}
Finally by Lemma \ref{lemma:existenceTtilde}, $\mathcal H F \in \mathbb L^1_{\mathcal P}$ and
$$ \lim_{k\to+\infty} \E\left[\left|\int_{[0,T]\times \X}  \mathcal{H}_{(t,x)}F^k \bN(dt,dx) - \int_{[0,T]\times \X}  \mathcal{H}_{(t,x)} F \bN(dt,dx)\right|\right]=0$$
leading to 
$\quad F = F(\bO_\emptyset) + \int_{[0,T]\times \X}  \mathcal{H}_{(t,x)} F \bN(dt,dx). $
\end{proof}
\noindent We now extend the previous result to the case  $\br(\bX)=+\infty$.
\begin{theorem}
\label{t{th:pseudoCOinfinite}}
Assume $\br(\bX)=+\infty$. Let $F$ in $L^1(\Omega)$. Then
$$ F = F(\bO_\emptyset) + \int_{\real_+ \times \X} \mathcal{H}_{(t,x)} F \; \bN(dt,dx), \; \P-a.s$$
and for any subset $\mathbf R \subset \bX$ with $\br(\mathbf R)<+\infty$, $\mathcal H F \ind{\mathbf R} \in \mathbb L^1_{\mathcal P}$ . 
In addition if there exists $Z \in  \mathbb L^0_{\mathcal P}$ such that 
$$ \E\left[\left|\int_{\real_+ \times \X} Z_{(t,x)} \; \bN(dt,dx)\right|\right] <+\infty$$
and
$$ F = F(\bO_\emptyset) + \int_{\real_+ \times \X}  Z_{(t,x)}   \; \bN(dt,dx), \; \P-a.s.$$
then for a.-e. $(t,x) \in \real_+ \times \X$, $Z_{(t,x)} = \mathcal{H}_{(t,x)}F$, $\P$-a.s..
\end{theorem}
\begin{proof}
The proof is divided in two parts. \\
\textbf{Uniqueness:} First note that contrary to Theorem \ref{th:pseudoCOfinite} the value of the random variable $F$ on the set $\bN(\real_+ \times \X)=0$ is not uniquely defined as $\P[\bN(\real_+ \times \X)=0]=0$. 
The same  proof as in Theorem \ref{th:pseudoCOfinite} gives that for  a.e. $(t,x) \in \real_+\times \X$, $\mathcal{H}_{(t,x)} F =Z_{(t,x)}$, $\P$-a.s..\\
\textbf{Existence:} 
{As $\br$ is a   non-atomic $\sigma$-finite measure on  $(\bX,\mathfrak X)$,} there exists a family of sets $(\mathbf R_j)_j$ such that for any $j$, $\br(\mathbf R_j) <+\infty$, $\mathbf R_j \subset \mathbf R_{j+1}$ and $\cup_j \mathbf R_j = \bX$. For any $j$ consider 
$$\Omega_j:=\left\{\bO=\sum_{j=1}^{n} \delta_{\bx_j}, \; \bx_j \in \mathbf R_j, \;j=1,\ldots,n, \; n\in \mathbb{N}\cup\{+\infty\} \right\} \subset \Omega$$
so that $\Omega_j \nearrow \Omega$ when $j$ tends to $+\infty$.
Let $F \in L^1(\Omega)$. As $\P[\bN(\bX)=0]=0$, once again we set $F(\bO_\emptyset):=0$.
Fix $j\geq 1$ and consider $F_j := F \ind{\Omega_j}$ so that $F^j \in L^1(\Omega)$. Writing $\bN^j$ the projection of $\bN$ on $\mathbf R_j$ (that is $\bN^j(A\times B):=\bN((A\times B)\cap \mathbf R_j)$) places ourselves  in the framework of Theorem \ref{th:pseudoCOfinite} and thus  
$$ F_j = F(\bO_\emptyset) + \int_{\mathbf R_j} \mathcal H_{(t,x)} F_j \bN^j(dt,dx), \; \P-a.s.$$
and $\mathcal H F_j \ind{\mathbf R_j} \in \mathbb L^1_{\mathcal P}$.
By definition $\bN^j$ coincides with $\bN$ on $\mathbf R_j$ and for $(t,x)$ in $\mathbf R_j$ we have for any $\bO \in \Omega_j$
\begin{align*}
(\mathcal H_{(t,x)} F_j)(\bO) &= (\bD_{(t,x)} F^j \circ \tau_t)(\bO)
=F^j(\bO_t + \delta_{(t,x)})-F^j(\bO_t)\\
&=F(\bO_t + \delta_{(t,x)}) \ind{\bO_t + \delta_{(t,x)} \in \Omega_j}-F(\bO_t) \ind{\bO_t \in \Omega_j}\\
&=(\bD_{(t,x)} F \circ \tau_t)(\bO) \ind{\bO_t \in \Omega_j} = (\mathcal H_{(t,x)} F)(\bO) \ind{\bO_t \in \Omega_j}.
\end{align*}
So $ \ind{\Omega_j} \mathcal H_{(t,x)} F_j = \ind{\Omega_j} \mathcal H_{(t,x)} F $
and thus
$ F_j = \int_{\mathbf R_j} \mathcal H_{(t,x)} F \, \bN^j(dt,dx), \; \P-a.s..$\\
Note also that by definition of $\bN^j$, 
$F_j = \int_{\real_+\times\X} \mathcal H_{(t,x)} F \, \bN^j(dt,dx), \; \P-a.s..$\\
By monotone convergence, $F^j$ converges to $F$ in $L^1(\Omega)$ so 
$$ \lim_{j\to+\infty} \E\left[\left|\int_{\real_+ \times \bX} \mathcal H_{(t,x)} F \bN^j(dt,dx) - \int_{\real_+ \times \bX} \mathcal H_{(t,x)} F \bN(dt,dx)\right|\right] =0$$
proving that $\int_{\real_+ \times \X} \mathcal H_{(t,x)} F \, \bN(dt,dx) \in L^1(\Omega)$.\\
\end{proof}

\begin{remark}  For the case of point processes, the Pseudo-Clark-Ocone formula  stated in Theorem \ref{th:pseudoCOfinite} coincides with the  Poisson imbedding representation   \eqref{imbedding}. Indeed   let 
 $H$ be a point process with  predictable intensity $\lambda$ defined through the thinning procedure 
\begin{equation*}
H_T = \int_{(0,T]\times \real_+} \ind{\theta\leq \lambda_t} \bN(dt,d\theta); \quad \forall T>0. 
\end{equation*}
Then 
 $H_T(\bO_\emptyset)= 0$,  and  using the previous representation of $H_T$, the difference operator  is
 \begin{align*}
 \bold{D}_{(t,x)} H_T    &= H_T (\cdot + \delta_{(t,x)}) - H_T(\cdot)\\
&= \ind{x \leq \lambda_t}  + \int_t^T   ( \ind{\theta \leq \lambda_r \circ \varepsilon^+_{(t,x)}} -  \ind{\theta \leq \lambda_r}  ) {\bN}(dr,d\theta).
\end{align*}
Therefore $\bold{D}_{(t,x)} H_T \circ \tau_t=    \ind{x \leq \lambda_t}$  which is $\mathcal F_{t^-}$-mesurable. Thus   \eqref{imbedding} re-writes as 
$$ H_T =  H_T(\bO_\emptyset) + \int_{[0,T] \times [0,\infty[} \mathcal{H}_{(t,x)} H_T \; \bN(dt,dx), \; \P-a.s..$$
Note that  the standard Clark-Ocone formula  is much more tricky to compute, since one has to compute the conditional expectation (given  $\mathcal F_{t^-}$) of the integral term 
 \begin{align*}
  H_T &=  \E(H_T) + \int_{[0,T] \times [0,\infty[}   \left( \E_{t-}\left[\ind{x \leq \lambda_t} + \int_t^T   ( \ind{\theta \leq \lambda_r \circ \varepsilon^+_{(t,x)}} -  \ind{\theta \leq \lambda_r}  ) {\bN}(dr,d\theta) \right]    \right)   \tilde{\bN}(dt,dx)\\
& =  \E\left(\int_{[0,T]}  \lambda_t dt \right) + \int_{[0,T] \times [0,\infty[}  \ind{x \leq \lambda_t}   \tilde{\bN}(dt,dx)\\
& \; \;   \; \;   \; \; +  \int_{[0,T] \times [0,\infty[}   \left( \E_{t-}\left[ \int_t^T   ( \ind{\theta \leq \lambda_r \circ \varepsilon^+_{(t,x)}} -  \ind{\theta \leq \lambda_r}  ) {\bN}(dr,d\theta) \right]    \right)   \tilde{\bN}(dt,dx).
\end{align*}

\end{remark}

\medskip
The following section gathers some technical lemmata.
\section{Technical  lemmata}\label{sseclemmata}
We first  prove convergence results of the operator  $\mathcal{H}_{(t,x)}$. More precisely:
\begin{lemma}
\label{lemma:existenceTtilde}
Assume  $\br(\bX)<+\infty$. Let $r \in [1,2]$ and $F \in L^{r}(\Omega)$. Assume there exists $(G_k)_k \subset L^{r}(\Omega)$ converging in $L^r(\Omega)$ to $F$ such that for any $k$, $\mathcal{H} G_k$ belongs to $\mathbb L^r_{\mathcal P}$. Then 
$$ \lim_{k\to+\infty} \left(\int_{[0,T]\times \X} \E\left[\left|\mathcal{H}_{(t,x)}G_k-\mathcal{H}_{(t,x)} F\right|^r\right] dt \pi(dx)\right)^{1/r}=0.$$
In particular, the stochastic process $\mathcal{H} F$ belongs to $\mathbb L^r_{\mathcal P}$ and we have the following convergence in $L^1$ 
$$ \lim_{k\to+\infty} \E\left[\left|\int_{[0,T]\times \X}  \mathcal{H}_{(t,x)}G_k \, \bN(dt,dx) - \int_{[0,T]\times \X}  \mathcal{H}_{(t,x)} F  \, \bN(dt,dx)\right|\right]=0.$$
\end{lemma}

\begin{proof}
We adapt the proof of \cite[Lemma 2.3]{Last:2011aa} and give only the main arguments. By definition of the Malliavin derivative (see Definition \ref{definition:Dn}) and Lemma \ref{lemma:conditionning}.
\begin{align*}
&\left(\int_{[0,T]\times \X} \E\left[\left|\mathcal H_{(t,x)} G_k- \mathcal H_{(t,x)} F\right|^r\right] dt \pi(dx)\right)^{1/r} \\
&=\E\left[\int_{[0,T]\times \X} \left|\E_{t-}\left[L^{T,t}\bD_{(t,x)} G_k\right]- \E_{t-}\left[L^{T,t}\bD_{(t,x)} F\right]\right|^r dt \pi(dx)\right]^{1/r} \\
&= \E\left[\int_{[0,T]\times \X} \left|\E_{t-}\left[L^{T,t}(\bD_{(t,x)} G_k-\bD_{(t,x)} F)\right]\right|^r dt \pi(dx)\right]^{1/r} \\
&\leq \E\left[\int_{[0,T]\times \X} \left|L^{T,t}(\bD_{(t,x)} G_k-\bD_{(t,x)} F)\right|^r dt \pi(dx)\right]^{1/r} \\
&= \E\left[\int_{[0,T]\times \X} (L^{T,t})^r \left|\bD_{(t,x)} G_k-\bD_{(t,x)} F\right|^r dt \pi(dx)\right]^{1/r} \\
&\leq C_r \E\left[\int_{[0,T]\times \X} (L^{T,t})^r \left|G_k\circ \varepsilon^+_{(t,x)}-F\circ \varepsilon^+_{(t,x)}\right|^r dt \pi(dx)\right]^{1/r} \\
&+ C_r \E\left[\int_{[0,T]\times \X} (L^{T,t})^r \left|G_k-F\right|^r dt \pi(dx)\right]^{1/r},
\end{align*}
where $C_r>0$ is a combinatorial constant depending only on $r$.  
Using Mecke's formula (see Lemma \ref{lemma:Mecke}) 
\begin{align*}
&\E\left[\int_{[0,T]\times \X} (L^{T,t})^r \left|G_k\circ \varepsilon^+_{(t,x)}-F\circ \varepsilon^+_{(t,x)}\right|^r dt \pi(dx)\right]^{1/r}  \\
&=\E\left[\int_{[0,T]\times \X} \exp(r(T-t)\pi(\X)) |F\circ \varepsilon^+_{(t,x)}-G_k\circ \varepsilon^+_{(t,x)}|^r \ind{\bN([t,T]\times \X)=0}   dt \pi(dx)\right]^{1/r}\\
&\leq \exp(T\pi(\X)) \E\left[\int_{[0,T]\times \X} |F\circ \varepsilon^+_{(t,x)}-G_k\circ \varepsilon^+_{(t,x)}|^r \ind{\bN((t,T]\times \X)=0}   dt \pi(dx)\right]^{1/r}\\
&= \exp(T\pi(\X)) \E\left[\int_{[0,T]\times \X} |F\circ \varepsilon^+_{(t,x)}-G_k\circ \varepsilon^+_{(t,x)}|^r \ind{\bN((t,T]\times \X)=0}\circ \varepsilon^+_{(t,x)}   dt \pi(dx)\right]^{1/r}\\
&= \exp(T\pi(\X)) \E\left[\int_{[0,T]\times \X} |F-G_k|^r \ind{\bN((t,T]\times \X)=0}  \bN(dt,dx)\right]^{1/r}\\
&= \exp(T\pi(\X)) \E\left[|F-G_k|^r \int_{[0,T]\times \X} \ind{\bN((t,T]\times \X)=0}  \bN(dt,dx)\right]^{1/r}.
\end{align*}
Fix $\bO = \sum_{j=1}^{n} \delta_{{(t_j,x_j)}}$, $(t_j,x_j)\in [0,T]\times \X$ and $n\in \mathbb N$ (as $\pi(\X)<+\infty$, $\bO([0,T]\times \X)<+\infty$ $\P$-a.s.). We have
\begin{align*}
&\int_{[0,T]\times \X}\ind{\bN((t,T]\times \X)=0} \bN(dt,dx)(\bO)\\
&=\sum_{j=1}^{n} \ind{\bO((t_j,T]\times \X)=0} \\
&=\ind{\bO((t_n,T]\times \X)=0} =\ind{\bO((0,T]\times \X) > 0} .
\end{align*}
Hence
\begin{align*}
&\E\left[\int_{[0,T]\times \X} (L^{T,t})^r \left|G_k\circ \varepsilon^+_{(t,x)}-F\circ \varepsilon^+_{(t,x)}\right|^r dt \pi(dx)\right]^{1/r}  \\
&\leq \exp(rT\pi(\X)) \E\left[|F-G_k|^r \ind{\bN((0,T]\times \X) > 0}\right]^{1/r}\\
&\leq \exp(rT\pi(\X)) \E\left[|F-G_k|^r\right]^{1/r} \to_{k\to +\infty} 0.
\end{align*}
Similarly
\begin{align*}
&\E\left[\int_{[0,T]\times \X} (L^{T,t})^r \left|G_k-F\right|^r dt \pi(dx)\right]^{1/r} \\
&=\E\left[\left|G_k-F\right|^r \int_{[0,T]\times \X} (L^{T,t})^r dt \pi(dx)\right]^{1/r} \\
&\leq \exp(T\pi(\X)) (T \pi(\X))^{1/r} \E\left[\left|G_k-F\right|^r\right]^{1/r} \to_{k\to +\infty} 0.
\end{align*}
In addition as $\mathcal H G_k$ converges in $\mathbb L^r_{\mathcal P}$ it converges pointwise and thus $\mathcal H F$ belongs to $\mathbb L^r_{\mathcal P}$. 
Furthermore 
\begin{align*}
&\E\left[\left|\int_{[0,T]\times \X}  \mathcal{H}_{(t,x)}G_k \bN(dt,dx) - \int_{[0,T]\times \X}  \mathcal{H}_{(t,x)} F \bN(dt,dx)\right|\right]\\
&\leq\E\left[\int_{[0,T]\times \X}  \left|\mathcal{H}_{(t,x)}G_k-\mathcal{H}_{(t,x)} F\right| \bN(dt,dx)\right] \\
&=\E\left[\int_{[0,T]\times \X}  \left|\mathcal{H}_{(t,x)}G_k-\mathcal{H}_{(t,x)} F\right| dt \pi(dx)\right] \\
&\leq (T \pi(\X))^{\frac{r-1}{r}} \E\left[\int_{[0,T]\times \X}  \left|\mathcal{H}_{(t,x)}G_k-\mathcal{H}_{(t,x)} F\right|^r dt \pi(dx)\right]^{1/r} \to_{p\to+\infty} 0
\end{align*}
using H\"older inequality.
\end{proof}

Finally, we prove that for any $n\in \mathbb N$ the space generated by the iterated integrals $\bold{I}_n(f_n)$ and $\bpI_n(f_n) $ (for   $f_n \in L^r(\bX^n)$) are identical.
\begin{lemma}
\label{lemma:span}
Assume $\br(\bX)<+\infty$ and let $r\geq 1$. 
$$ \textrm{Span}\left\{ \bold{I}_k(f_k); \; k \in [\![ 0, n ]\!],  \;  f_k \in L^r(\bX^k)\right\} = \textrm{Span}\left\{ \bpI_k(f_k);  \; k \in [\![ 0, n ]\!],  \;  f_k \in L^r(\bX^k)\right\}$$

where $\bold{I}_0(f):=f$ and $\bpI_0(f):=f$ for $f\in \mathbb R$. In addition for any $n\in \mathbb N^*$ and $f_n \in L^r(\bX^n)$,
\begin{equation}\label{CO1}
 \bold{I}_n(f_n) = \int_{[0,T]\times \X} {\E_{t-}[\bD_{(t,x)} \bold{I}_n(f_n)]} \tilde{\bN}(dt,dx), \end{equation}
\begin{equation} \label{CO2}
 \bpI_n(f_n) =  \int_{[0,T]\times \X} {\mathcal H_{(t,x)} \bpI_n(f_n)} \bN(dt,dx),
 \end{equation}
 and $\mathcal H \bpI_n(f_n)$ belongs to $\mathbb L^1_{\mathcal P}$. Finally if $r=2$, then ${\E_{\cdot-}[\bD \, \bold{I}_n(f_n)]} = (\bD \, \bold{I}_n(f_n))^p$ and it belongs to $ \mathbb L^2_{\mathcal P}$.
\end{lemma}

\begin{proof}
Assume $\br(\bX)<+\infty$. Recalling that $\bold{I}_n(f_n) = \bold{I}_n(\tilde f_n)$ and $\bpI{I}_n(f_n) = \bpI_n(\tilde f_n)$ where $\tilde f_n$ denotes the symmetrization of $f_n$, without loss of generality we assume that all the mappings below are symmetric.\\
\underline{Step 1:}
Let $n\in \mathbb N^*$ and $F:=\bpI_n(f_n)$ with $f_n \in {L^r_s(\bX^n)}$, $r\geq 1$. Then according to Proposition 12.11 in \cite{Last:2011aa}, $F = \sum_{j=0}^n \bold{I}_j(g_j)$ with 
$$g_j(\bx_1,\ldots,\bx_{j}):=\frac{n!}{j! (n-j)!} \int_{\bX^{n-j}} f_n(\by_1,\ldots,\by_{n-j},\bx_1,\ldots,\bx_{j}) \br^{\otimes (n-j)}(d\by_1,\ldots,d\by_{n-j}) \in L^r(\bX^{n-j})$$
as
\begin{align*}
&\|\int_{\bX^{n-j}} f_n(\by_1,\ldots,\by_{n-j},\cdot) \br^{\otimes (n-j)}(d\by_1,\ldots,d\by_{n-j})\|_{L^r(\bX^{n-j})} \\
&\leq \int_{\bX^{n-j}} \| f_n(\by_1,\ldots,\by_{n-j},\cdot)\|_{L^r(\bX^{n-j})} \br^{\otimes (n-j)}(d\by_1,\ldots,d\by_{n-j}) \\
&= \int_{\bX^{n-j}} \left(\int_{\bX^{j}} |f_n(\by_1,\ldots,\by_{n-j},\bx_1,\ldots,\bx_{j})|^r \br^{\otimes (n-j)}(d\bx_1,\ldots,d\bx_j)\right)^{1/r} \br^{\otimes (n-j)}(d\by_1,\ldots,d\by_{n-j}) \\
&\leq (T\pi(\X))^{\frac{r}{r-1}} \|f_n\|_{L^r(\bX^n)}.
\end{align*}
Let $n\in \mathbb N^*$ and $F:=\bold{I}_n(f_n)$ with $f_n \in L^r(\bX^n)$, $r\geq 1$. Then 
\begin{align*}
F &=\sum_{J\subset \{1,\ldots,n\}} (-1)^{n-|J|} \int_{\bX^{n-|J|}} \int_{\bX^{|J|}} f_n(\bx_1,\ldots,\bx_n) \; \bN^{(|J|)}(d\bx_J) \br^{\otimes(n-|J|)}(d\bx_{J^c})\\
&= (-1)^{n} \int_{\bX^{n}} f_n(\bx_1,\ldots,\bx_n) \br^{\otimes n}(d\bx_1,\ldots,d\bx_n)\\
&+ \sum_{J\subset \{1,\ldots,n\}; J \neq \emptyset} (-1)^{n-|J|} \int_{\bX^{n-|J|}} \int_{\bX^{|J|}} f_n(\bx_1,\ldots,\bx_n) \; \bN^{(|J|)}(d\bx_J) \br^{\otimes (n-|J|)}(d\bx_{J^c})\\
&= (-1)^{n} \int_{\bX^{n}} f_n(\bx_1,\ldots,\bx_n) \br^{\otimes n}(d\bx_1,\ldots,d\bx_n)\\
&+ \sum_{k=1}^{n} \int_{\bX^{k}}  \frac{n!}{k!(n-k)!} (-1)^{n-k} \int_{\bX^{n-k}} f_n(\by_1,\ldots,\by_{n-k},\bx_1,\ldots,\bx_k)\br^{\otimes (n-k)}(d\by) \bN^{(k)}(d\bx) \\
&= \bpI_0(g_0) + \sum_{k=1}^{n}  \bpI_k(g_k)
\end{align*}
with $g_0:=(-1)^{n} \int_{\bX^{n}} f_n(\bx_1,\ldots,\bx_n) \br^{\otimes n}(d\bx_1,\ldots,d\bx_n)$ and as {above}
$$ g_k := \frac{n!}{k!(n-k)!} (-1)^{n-k} \int_{\bX^{n-k}} f_n(\by_1,\ldots,\by_{n-k},\cdot)\br^{\otimes(n-k)}(d\by) \in L^r(\bX^{k})$$
Then step1 together with step 2 implies that
$$ \textrm{Span}\left\{ \bold{I}_k(f_k); \; k \in [\![ 0, n ]\!],  \;  f_k \in L^r(\bX^k)\right\} = \textrm{Span}\left\{ \bpI_k(f_k);  \; k \in [\![ 0, n ]\!],  \;  f_k \in L^r(\bX^k)\right\}.$$
\underline{Step 3:}
Let $n\in \mathbb N^*$ and $f_n\in {L^r_s(\bX^n)}$. We have that for any $(t,x) \in \bX$,
$$ \bD_{(t,x)} \bold{I}_n(f_n) = n \bold{I}_{n-1}(f_n((t,x),\cdot))$$
and thus 
$$ \E_{t-}[\bD_{(t,x)} \bold{I}_n(f_n)] = n \bold{I}_{n-1}\left(f_n((t,x),\cdot) \ind{[0,t)\times \X}^{\otimes (n-1)}\right),$$
with {$\ind{[0,t)\times \X}^{\otimes (n-1)}((t_2,x_2),\ldots,(t_n,x_n)):=\prod_{i=2}^n \ind{t_i <t}$.}
The Clark-Ocone representation \eqref{CO1} follows from classical Malliavin's calculus. Representation \eqref{CO2} follows from the definition of integrals $\bpI$ involving the non-compensated Poisson measure $\bN$. Thus uniqueness part of the proof gives that \eqref{CO2} is in force. The last thing to be proved is that both integrands in \eqref{CO1}  and \eqref{CO2}   are predictable. \\
\\
\underline{Step 4:}
 Let $n\geq 1$ and  $f_n \in L^r(\bX^n)$. Using classical approximations (as $\br(\bX)<+\infty$) there exists a sequence $(f_n^{\ell})_{\ell \geq 1} \subset C_b(\bX)$ (the set of bounded and continuous functions on $\bX$) such that $f_n^{\ell} \to_{\ell \to +\infty} f_n$ in $L^r(\bX^n)$. Set $F^\ell :=\bpI_n(f_n^{\ell})$ {where by abuse of notation we make use of the same notation of $f_n^\ell$ for its symmetrization}. We have that 
$$ \E\left[\left|F^\ell - F\right|\right] \leq \int_0^T \int_{\X} \E\left[\left|n \bpI_{n-1}(f_n^{\ell}((t,x),\cdot)) - n \bpI_{n-1}(f_n((t,x),\cdot))\right|\right] \pi(dx)dt \leq \|f_n^\ell - f_n\|_{\mathbb L^1_{\mathcal P}} \to 0.$$    
Thus,
$$ \mathcal H_{(t,x)} F^\ell = n \bpI_{n-1}(f_n^{\ell}((t,x),\cdot)\ind{[0,t)\times \X}^{\otimes (n-1)}) $$
is left continuous in $t$ for any $x$, since $f_n^{\ell}$ is a continuous function. Hence we have that $\mathcal H F^\ell $ belongs to $\mathbb L^1_{\mathcal P}$ and so is $\mathcal H F$ by Lemma \ref{lemma:existenceTtilde}.\\\\
\noindent
Similarly if $r=2$, thanks to the isometry property of the operator $ \bold{I}_n$ (see the comment at the end of Definition \ref{definition:interated})
\begin{align*}
&\int_0^T \int_{\X} \E\left[\left|\E_{t-}[\bold{I}_{n}(f_n)]-\E_{t-}[\bold{I}_{n}(f_n^\ell)]\right|^2\right] \pi(dx) dt\\
&=\int_0^T \int_{\X} \E\left[| n \bold{I}_{n-1}\left(f_n((t,x),\cdot)\ind{[0,t)\times \X}^{\otimes (n-1)} \right) - n  \bold{I}_{n-1}\left(f_n^l((t,x),\cdot)\ind{[0,t)\times \X}^{\otimes (n-1)} \right) |^2\right] \pi(dx) dt\\
&=n^2 \|f_n-f_n^\ell\|^2_{L^2(\bX^n)} \to_{\ell\to +\infty} 0.
\end{align*}
Once again the continuity of the maps $f_n^\ell$ gives that $\E_{t-}[\bold{I}_{n}(f_n)]$ is the predictable projection of $\bD \bold{I}_n(f_n)$ and thus that it belongs to $\mathbb L^2_{\mathcal P} $.
\end{proof}

\subsection*{Conclusion}
{This paper extends  to any integrable Poisson-functional  the Poisson imbedding that  provides a representation of any point process  (with intensity  $(\lambda_t)_t$) as the integral of  $\ind{\theta\leq \lambda_t}$  with respect to an uncompensated Poisson measure $\bN(dt,d\theta)$. More precisely, we provide  for any   $F\in  L^1(\Omega)$  a Pseudo-Clark-Ocone representation formula  with respect to the uncompensated Poisson measure and whose  integrand is fully characterized as $ \mathcal{H}_{(t,x)} F =\bD_{(t,x)}F \circ \tau_t $. Besides $\mathcal H F$ is well defined in  $ \mathbb L^1_{\mathcal P}$, contrary to the standard Clark-Ocone formula which requires  $F \in L^2(\Omega)$ and whose integrand $(\bD_{(t,x)}F)^p$ is only defined as a limit in $L^2(\Omega\times \bX)$. }

\section{Appendix}\label{Appendix}
For sake of completeness, we provide below some technical elements on the standard Clark-Ocone formula.
\begin{lemma}
Consider $X$ a $ \mathcal F^{\bN}\otimes \mathfrak{X}$-measurable process with $\int_{[0,T]\times \X}\E[|X_{(t,x)}|]dt\pi(dx)<+\infty$. Then  there exists a unique element $X^p\in\mathbb L^1_{\mathcal P} $ such that for any predictable stopping time $\tau$
$$
\E\left[ X_{\tau,x} |\mathcal F_{\tau -}^{\bN} \right]\ind{\tau<\infty} = X^p_{\tau,x}\ind{\tau<\infty}.
$$
\end{lemma}
\begin{proof}
The proof essentially follows Lemma 3.3 in \cite{zhang2009clark}, in which $X$ is assumed to be bounded.
Let $X$ such that $\int_{[0,T]\times \X}\E[|X_{(t,x)}|]dt\pi(dx)<+\infty$ and we set $X^k := -k \vee X \wedge k$ for $k\in \mathbb N^*$. For any $k \geq 1$, there exists a unique $\mathcal P$-measurable process $Y^k$ such that for any predictable stopping $\tau$ and $x\in\X$\\
$$
\E\left[ X^k_{\tau,x} |\mathcal F_{\tau -}^{\bN} \right]\ind{\tau<\infty} = Y^k_{\tau,x}\ind{\tau<\infty} \; \mathbb P-a.s..
$$
By monotone convergence for any $x \in\X$,
$$
\E\left[ X_{\tau,x} |\mathcal F_{\tau -}^{\bN} \right]\ind{\tau<\infty} = Y_{\tau,x}\ind{\tau<\infty} \; \mathbb P-a.s.,
$$
where $Y:= \lim_{k\to +\infty}Y^k$ belong to $\mathbb L^1_{\mathcal P}$.
\end{proof}

We now  give a concise  proof of the  standard Clark-Ocone formula  for  any $F \in  L^2(\Omega)$ (Theorem \ref{th:CO}). In particular, following \cite[Section 3.2]{Privault_2009},  we  make precise the definition of the integrand $\left({ \bD F} \right)^p$ in $L^2(\Omega\times \bX)$ in a limiting procedure and based on the continuous feature of the mapping $F \mapsto \left(\E_{t-}\left[\bD_{(t,x)} F\right]\right)_{(t,x)}$ together with the fact that for fixed $(t,x)$, $\E_{t-}\left[\bD_{(t,x)} F\right] = \left({ \bD_{(t,x)} F} \right)^p \; \mathbb P-a.s.$.

\begin{proof}
The chaotic expansion (Theorem \ref{th:chaoticgeneral}) entails that $V:=\textrm{Span}\left\{ \bold{I}_n(f_n); \quad n\in \mathbb N, \;  f_n \in L^2(\bX)\right\}$ is dense in $L^2(\Omega)$. Let $F:=\bold{I}_n(f_n) \in V$ with $n>0$ and without loss of generality $f_n$ is assumed to be symmetric. We have that 
$$ \bD_{(t,x)} F = n \bold{I}_{n-1}(f_n((t,x),\cdot)) $$
and by Lemma \ref{lemma:span}
$$ (\bD_{(t,x)} F)^p =  \E_{t-}[\bD_{(t,x)} F] = n \bold{I}_{n-1}(f_n((t,x),\cdot)\ind{([0,t)\times \X)^{\otimes (n-1)}}(\cdot)) $$
so that
\begin{equation}
\label{eq:COtemp1}
F = \int_{[0,T] \times \X} \E_{t-}\left[\bD_{(t,x)} F\right] \tilde{\bN}(dt,dx),
\end{equation}
and 
\begin{align*}
&\E\left[\int_{[0,T]\times \X} \left|\E_{t-}\left[\bD_{(t,x)} F\right]\right|^2 dt\pi(dx)\right]\\
&= n^2 \int_{[0,T]\times \X}  \E\left[\left|\bold{I}_{n-1}(f_n((t,x),\cdot)\ind{([0,t)\times \X)^{\otimes (n-1)}}(\cdot))\right|^2 \right]dt\pi(dx)\\
&= n (n-1)! \int_{([0,T]\times \X)^n} \left|f_n\right|^2 (dt\pi(dx))^{\otimes n} \\
& =\E[|F|^2] <+\infty.
\end{align*}
Note that by orthogonality of the operators $\bold{I}_n$ the previous result extends to any element $F$ on $V$ with $\E[F]=0$, that is for any $F\in V$ with $\E[F]=0$, 
$$\E\left[\int_{[0,T]\times \X} \left|\E_{t-}\left[\bD_{(t,x)} F\right]\right|^2 dt\pi(dx)\right]=\E[|F|^2] <+\infty.$$
Let the operators $U$ and $\tilde U$ defined as 
$$ 
U:\begin{array}{ll}
V &\to \mathbb L^2_{\mathcal P}\times \real_+ \\
F & \mapsto \left(\left(\E_{t-}[\bD_{(t,x)} F]\right)_{(t,x)}; \E[F]^2\right).
\end{array}
$$

$$ 
\tilde U:\begin{array}{ll}
V &\to \mathbb L^2(\Omega) \\
F & \mapsto \int_{[0,T]\times \X} \E_{t-}[\bD_{(t,x)} F] \tilde{\bN}(dt,dx).
\end{array}
$$
For any $F \in V$, we have proved that $F$ enjoys Representation (\ref{eq:COtemp1}) and Relation (\ref{eq:Isometry}) implies that
$$\|U(F)\|_{L^2(\mathbb L^2_{\mathcal P}\times \real_+)}^2= \|F\|_{L^2(\Omega)}^2; \quad \|\tilde U(F)\|_{L^2(\Omega)}^2 = \E[F^2]-\E[F]^2 \leq \|F\|_{L^2(\Omega)}^2.$$
Let $F \in L^2(\Omega)$. There exists $(F_n)_n \subset V$ such that $\lim_{n\to+\infty} \E[|F_n-F|^2]=0$. Using the previous isometries the sequences $(U(F_n))_n$ and $(\tilde U(F_n))_n$ are Cauchy and thus converging respectively in $L^2(\mathbb L^2_{\mathcal P}\times \real_+)$ and $L^2(\Omega)$. Let $U(F):=(U^1(F),U^2(F))$ and $\tilde{U}(F))$ their limits where $U^1$ and $U^2$ denote the two components of the limits. In addition, the $L^2(\Omega\times [0,T]\times \X)$-convergence of the process $U^1(F_n)$ implies that $U^1(F)$ is $\mathcal P$-measurable. Hence the process $t\mapsto \int_{[0,t]\times\X} U^1(F)_{(r,x)} \tilde{\bN}(dr,dx)$ is a martingale and 
\begin{align*}
\E\left[\left|\int_{[0,T]\times\X} U^1(F)_{(r,x)} \tilde{\bN}(dr,dx)\right|^2\right] &= \E\left[\int_{[0,T]\times\X} |U^1(F)_{(r,x)}|^2 dr \pi(dx)\right] \\
&= \lim_{n\to+\infty} \E\left[\int_{[0,T]\times\X} |U^1(F_n)_{(r,x)}|^2 dr \pi(dx)\right] \\
&= \lim_{n\to+\infty} \E[F_n^2]-\E[F_n]^2 = \E[F^2]-\E[F]^2.
\end{align*}
Thus
\begin{align*}
&\E\left[\left|F-\E[F]-\int_{[0,T]\times\X} U^1(F)_{(r,x)} \tilde{\bN}(dr,dx)\right|^2\right]\\
&\leq \lim_{n\to+\infty} (\E\left[\left|F-F_n\right|^2\right]+ |\E[F]-\E[F_n]|)\\
&+\lim_{n\to+\infty} \E\left[\left|\int_{[0,T]\times\X} U^1(F_n)_{(r,x)} \tilde{\bN}(dr,dx)-\int_{[0,T]\times\X} U^1(F)_{(r,x)} \tilde{\bN}(dr,dx)\right|\right]\\
&\leq \lim_{n\to+\infty} \E\left[\int_{[0,T]\times\X} |U^1(F_n)_{(r,x)}-U^1(F)|^2 dr \pi(dx)\right]\\
&\leq \lim_{n\to+\infty} \|U^1(F_n)_{(r,x)}-U^1(F)\|_{L^2(L^2_{\mathcal P}\times \real_+)}^2=0.
\end{align*}
The proof is complete. 
\end{proof}

\bibliographystyle{alpha-fr}
\bibliography{biblioHPR}

\end{document}